\documentclass[10pt,a4paper]{article}
\usepackage[latin1]{inputenc}
\usepackage{amsfonts}
\usepackage{amscd}
\usepackage{amsmath}
\usepackage{geometry}
\usepackage{color}
\usepackage{theorem}
\usepackage{epsfig}
\usepackage{mathrsfs}
\usepackage{pst-plot, pstricks}
\usepackage{fancybox,amssymb}
\usepackage[english]{babel}
\usepackage{comment}
\usepackage{slashbox}
\newcommand{\R}{\mathbb{R}}
\newcommand{\N}{\mathbb{N}}

\newcommand{\bq}{\begin{eqnarray*}}
\newcommand{\eq}{\end{eqnarray*}}

\newcommand{\erw}{\mathbb{E}} 
\makeatletter\def\blfootnote{\xdef\@thefnmark{}\@footnotetext}\makeatother
\theoremstyle{break}
\newtheorem{defn}{Definition}[section]
\newtheorem{theorem}[defn]{Theorem}
\newtheorem{proposition}[defn]{Proposition}
\newtheorem{kor}[defn]{Corollary}

\newtheorem{exmp}[defn]{Example}
\newtheorem{remark}{Remark}[section]

\makeatletter
\newenvironment{proof}{\noindent{\textit{Proof:}}}{%
\unskip\nobreak\hfil\penalty50\hskip1em\null\nobreak
$\Box$
\parfillskip=\z@\finalhyphendemerits=0\endgraf\bigskip} 
\let\oldendexmp\endexmp
\def\endexmp{\unskip\nobreak\hfil\penalty50\hskip1em\null\nobreak\hfil%
$\blacksquare$\parfillskip=\z@\finalhyphendemerits=0\endgraf\oldendexmp}
\makeatother

\title{{\Large Distribution functions, extremal limits and optimal transport}\\
{\large Dedicated to the $125$th anniversary of J.G. van der Corput}
}
\author{\normalsize M.R. Iac\`{o}, R.F. Tichy and S. Thonhauser\thanks{The authors
are supported by the Austrian Science Fund (FWF) Project F5510 (part of the Special Research
Program (SFB) \textquotedblleft Quasi-Monte Carlo Methods: Theory and Applications\textquotedblright). The first author is also partially supported by the Austrian Science Fund (FWF): W1230, Doctoral Program \textquotedblleft Discrete Mathematics\textquotedblright 
}}
\date{}
\newcommand{\Addresses}{{
\bigskip
\footnotesize
\textsc{Institute of Analysis and Computational Number Theory (Math A), Graz University of Technology,
Steyrergasse 30/II, 8010 Graz, Austria}\par\nopagebreak
\textit{E-mail addresses}: \texttt{iaco@math.tugraz.at, tichy@tugraz.at, stefan.thonhauser@math.tugraz.at}
}}

\begin{document}
\maketitle
\begin{abstract}
Encouraged by the study of extremal limits for sums of the form
$$\lim_{N\to\infty}\frac{1 }{N}\sum_{n=1}^N c(x_n,y_n)$$
with uniformly distributed sequences $\{x_n\},\,\{y_n\}$ the following extremal problem is of interest
$$\max_{\gamma}\int_{[0,1]^2}c(x,y)\gamma(dx,dy),$$
for probability measures $\gamma$ on the unit square with uniform marginals, i.e., measures whose distribution function is a copula.\\
The aim of this article is to relate this problem to combinatorial optimization and to the theory of optimal transport. 
Using different characterizations of maximizing $\gamma$'s one can give alternative proofs
of some results from the field of uniform distribution theory and beyond that treat additional questions.
Finally, some applications to mathematical finance are addressed.
\end{abstract}
\section{Introduction and motivation}
In a series of papers J.G. van der Corput \cite{van_der_Corput_1935, van_der_Corput_1936} systematically investigated distribution functions of sequences of real numbers. Some of his main results are as follows:
\begin{itemize}
\item[(i)] Any sequence of real numbers has a distribution function.
\item[(ii)] Any everywhere dense sequence of real numbers can be rearranged in such a way that the new sequence has an arbitrarily given distribution function.
\end{itemize}
Clearly, in general a distribution function is not uniquely determined by the sequence. Furthermore, van der Corput established necessary and sufficient conditions for a set $\mathcal{M}$ of non-decreasing functions such that $\mathcal{M}$ is the set of distribution functions of some sequence of real numbers.\\

More recently, the study of distribution functions was extended to multivariate functions by the Slovak school of O. Strauch and his coworkers; see \cite{Strauch_school_III, Strauch_school_IV, Strauch_school_II, Strauch_school_I}. In particular, they studied properties of the set of distribution functions of sequences in $[0,1]^2$ and various extremal problems related to distribution functions. It should be noted that bi-variate distribution functions are well-known in financial mathematics for modeling dependencies in risk processes.\\

From Fialov{\'a} \& Strauch \cite[Thm 1]{FialovaStrauch} one knows that for uniformly distributed sequences $\{x_n\},\,\{y_n\}$ in $[0,1]$
and a continuous function $c:[0,1]^2\to\R$ one has
\begin{equation}\label{eq:limit}
\lim_{N\to\infty}\frac{1}{N}\sum_{n=1}^N c(x_n,y_n)=\int_{[0,1]^2} c(x,y)\gamma(dx,dy),
\end{equation}
where $\gamma$ is a probability measure on the unit square equipped with the $\sigma-$algebra of Borel sets. Such a measure exhibits
a bi-variate distribution function $C:[0,1]^2\to [0,1]$ which is generally called a \emph{copula}.
The aim of this article is to provide a connection between the problem of finding extremal limits in (\ref{eq:limit})
(or maximal and minimal bounds for such limits) by studying the optimization problem
\begin{equation}\label{problem}
\int_{[0,1]^2} c(x,y) \gamma(dx,dy)\,\mapsto\;\mbox{max}
\end{equation}
and the field of optimal transport. Indeed, we will show how this problem can be perfectly embedded in the general theory of optimal transport.\\
Motivated by the discussion on the limiting property \eqref{eq:limit} problem \eqref{problem} attracted some attention in the number theoretic community
and found its way on the collection of unsolved problems of \emph{Uniform Distribution Theory}
\footnote{Problem 1.29 in the open problem collection as of 28.\ November 2013 (http://www.boku.ac.at/MATH/udt/unsolvedproblems.pdf)}.
We will mention some existing results in that context below. Notice that in the uniform distribution literature, problem \eqref{problem} is
originally written as an optimization with respect to functions $C\colon [0,1]^2\rightarrow [0,1]$ satisfying the
following properties: for every $x,y \in [0,1]$
\begin{align*}
C(x,0) &= C(0,y) = 0,\\
C(x,1) &= x\;\text{and}\; C(1,y) = y,
\end{align*}
and for every $x_1,x_2,y_1,y_2 \in [0,1]$ with $x_2 \geq x_1$ and $y_2 \geq y_1$
\begin{equation*}
 C(x_2,y_2) - C(x_2, y_1) - C(x_1, y_2) + C(x_1,y_1) \geq 0.
\end{equation*}
Clearly, in this particular situation a copula is a bivariate distribution function on $[0,1]^2$ with standard uniform marginals.
From the above stated properties one additionally sees that a copula $C$ induces a (Borel-)probability measure $\gamma$ on $[0,1]^2$, via the formula
\begin{equation*}
\gamma([a,b]\times[c,d])=C(b,d)-C(b,c)-C(a,d)+C(a,c)\,.
\end{equation*}
A first result in the direction of extremal limits \eqref{eq:limit} is given in \cite{PS},
where the authors take $c(x,y)=|x-y|$ in order to find optimal upper and lower bounds on the average distance between consecutive
points of u.d. sequences. In particular, they proved $\lim_{N\to\infty}\frac{1}{N}\sum_{n=0}^{N-1}|x_{n+1} - x_n| = \frac{2(b-1)}{b^2}$ for the van der
Corput sequence $(\phi_b(n))_{n\geq 0}$ in base $b$.\\
By looking at the same problem, but in the formulation of \eqref{problem}, the authors in \cite{FialovaStrauch} could give an explicit
formula for the asymptotic distribution function of the sequence $(\phi_b(n), \phi_b(n+ 1))_{n\geq 0}$, that is of the copula $C(x,y)$.\\
The problem of finding the limit distribution of consecutive elements of the van der Corput sequence was also considered in \cite{AH},
but using a different approach based on ergodic properties of the sequence itself. However, this approach does not give an explicit
form of the copula $C(x,y)$. This last problem is not easy to handle with and, apart from the already mentioned papers \cite{FialovaStrauch, PS},
only \cite{FMS} is known, where the authors found an explicit asymptotic distribution function of the sequence
$(\phi_b(n), \phi_b(n+ 1), \phi_b(n+2))_{n\geq 0}$.\\
Problem \eqref{problem} has been recently studied in \cite{HoferIaco} in connection with a well-known problem in combinatorial optimization,
namely the linear assignment problem. By means of this tool, the authors give optimal upper and lower bounds for
integrals of two-dimensional, piecewise constant functions with respect to copulas and construct the copulas for which these
bounds are attained. More precisely, the copulas realizing these bounds are shuffles of $M$, where the permutation $\sigma$
is the one which solves the assignment problem.\\
From the uniform distribution point of view, this class of copulas represents a family of uniform distribution preserving
mappings (u.d.p.), i.e. maps $f$ generating u.d. sequences $(f(x_n))_{n\in\mathbb{N}}$ for every u.d. sequence $(x_n)_{n\in\mathbb{N}}$.
We will discuss more extensively the linear assignment problem and the algorithm that solves it in one of the following next sections.\\
\\
The problem \eqref{problem} is known as Monge-Kantorovich transport problem.
Its origin is the question of how to transport soil from one location to another at minimal costs.
More precisely, suppose these two locations are disjoint subsets $M$ and $F$ of the Euclidean plane $\mathbb{R}^2$
and that $c(x,y) \colon \mathbb{R}^2 \times \mathbb{R}^2 \rightarrow [0, \infty)$ is the cost of transporting one
shipment of soil from $x$ to $y$. For simplicity, we assume that there is no splitting of shipments. Thus, a
transport map is a function $T \colon M \rightarrow F$. The goal is to find the optimal transport map which minimizes the total costs 
\begin{equation*}
c(T) := \sum_{m \in M} c(m, T(m)),
\end{equation*}
under the restriction that all the soil needs to be moved.
We will give a precise description of the optimal transport problem and fundamental results in section \ref{transport}.\\ 
The Monge-Kantorovich problem has found a great variety of applications in pure and applied mathematics, such as Ricci curvature
\cite{ricci}, nonlinear partial differential equations \cite{diffequ}, gradient flows \cite{flow}, structure of cities \cite{city}, maximization of
profits \cite{profit}, leaf growth \cite{leaf} and so on.\\
\\
As a prominent field of application of the theory of copulas and the transport problem one needs to mention financial mathematics.
In the last decade copulas, or more precisely some parametrized families of copulas, became very popular in finance for
modelling dependence structures within groups of assets or more generally between different kinds of risk factors.
In the past the study of dependence structures was typically reduced to the determination of correlation coefficients.
However, correlation coefficients describe dependencies perfectly only in the situation of marginally normal distributed risk factors,
while distributions obtained from financial market data are typically not normal. A standard introduction to risk modeling
and particularly to practical aspects of copula modeling is the book by McNeil et al. \cite{Embrechtsbook}.\\
It turned out that a precise description of the dependency structure within a portfolio of risks is not feasible in practice.
That's why one is trying to determine some (one may call it worst case or robust) bounds on risk measures of portfolios.  
For this purpose it is possible to utilize variants of the optimal transport problem:
\emph{try to minimize a risk measure of a portfolio of several risks with respect to their distribution while preserving their marginals}.
Some recent publications studying problems from risk management are R{\"u}schendorf \cite{RueschenBook2013}, Puccetti \& R{\"u}schendorf \cite{PucRuesch2013}
or Bernard et al. \cite{BJW}. A paper dealing with model independent bounds on option prices using theory of optimal transport
is Beiglb{\"o}ck et al. \cite{Beigl2013}.

%
\section{Mathematical formulation}
In this section we give precise statements of the mathematical objects at hand. We refer to \cite{DrmotaTichy1997, DurSem15, Joe, Nelsen} for details. Our starting point is Sklar's Theorem
(see e.g. \cite[Theorem 3.2.2]{Nelsen}), a classical result about copulas which provides the theoretical foundation for application of copulas.
\begin{theorem}

Given a $d$-dimensional distribution function (d.f.) $H$ with marginals $F_1,\dots ,F_d$, there exists a $d$-copula $C$ such
that for all $(x_1, x_2, \dots , x_d ) \in \mathbb{R}^d$
\begin{equation}\label{sklar}
H(x_1, x_2, \dots , x_d ) = C (F_1(x_1), F_2(x_2),\dots , F_d (x_d))
\end{equation}
The copula $C$ is uniquely defined on $\prod_{j=1}^d Ran(F_j)$ and is therefore
unique if all the marginals are continuous (here $Ran(F_j)$ denotes the range of $F_j$).\\
Conversely, if $F_1, F_2,\dots, F_d$ are $d$ (1-dimensional) d.f.'s, then the
function $H$ defined through Eq. \eqref{sklar} is a $d$-dimensional d.f..
\end{theorem}
Given a $d$-variate d.f.\ $F$, one can derive a copula $C$. Specifically,
when the marginals $F_i$ are continuous, $C$ can be obtained by means
of the formula
\begin{equation*}
C(u_1, u_2,\dots, u_d ) = F(F_1^{-1}(u_1), F_2^{-1}(u_2),\dots, F_d^{-1}(ud ))\ ,
\end{equation*}
where $F_i^{-1}=\inf\{t | F_i(t)\geq s\}$ is the pseudo-inverse of $F_i$.\\
Thus, copulas are essentially a way for transforming the r.v. $(X_1, X_2,\dots, X_d)$
into another r.v. $(U_1,U_2,\dots, U_d) = (F_1(X_1), F_2(X_2),\dots, F_d(X_d))$
having uniform margins on $[0,1]$ and preserving the dependence among the components.
As we have seen in Section 1, every copula $C$ induces a probability measure $\gamma$. Moreover,
there is a one-to-one correspondence between copulas and doubly stochastic measures. For every copula $C$,
the measure $\gamma$ is doubly stochastic in the sense that for every Borel set
$B\subset [0,1]$,  $\gamma([0, 1] \times B) = \gamma (B \times [0, 1])= \lambda(B)$ where $\lambda$ is the Lebesgue measure on
$[0,1]$. Conversely, for every doubly stochastic measure $\mu$, there exists a copula $C$ given by
$C(u, v)=\mu(([0, u])\times([0, v]))$. Clearly, a probability measure on $([0,1]^2,\mathcal{B}([0,1]^2))$
with uniform marginals is doubly stochastic. Therefore, we can translate some measure-theoretic concepts
and results into the language of copulas.

In particular, we are interested in the correspondence between copulas $C$
and measure-preserving transformations $f,g$ on the unit interval, via the formula
\begin{equation*}
C_{f,g} (u, v) = \lambda (f^{-1} [0, u])\cap \lambda (g^{-1} [0, v])\ .
\end{equation*}
We refer to \cite{ACS} for details and the study of related properties.\\
The following theorem shows how every copula can be bounded from above and below. The upper and lower bounds are called Fr\'echet-Hoeffding bounds.
\begin{theorem}
Suppose $F_1,\dots, F_d$ are marginal d.f.'s and $F$ is
any joint d.f. with those given marginals, then for all $\mathbf{x}\in \mathbb{R}^d$,
\begin{equation}\label{frechet}
 \left(\sum_{k=1}^d F_k(x_k) + 1 - d\right)^+\leq F(x) \leq \min (F_1 (x_1),\dots, F_d (x_d))\ .
\end{equation} 
\end{theorem}
The right-hand side of \eqref{frechet} is always a copula, whereas the left-
hand side is a copula only $d = 2$, see \cite[Theorem 3.2 and 3.3]{Joe}.\\
Thus, the problem is to find bounds of the form
\begin{equation}\label{bounds}
  \int_{[0,1[^2} c(x,y) dC_{\min}(x,y) \leq  \int_{[0,1[^2} c(x,y) dC(x,y) \leq  \int_{[0,1[^2} c(x,y) dC_{\max}(x,y),
\end{equation}
where $C_{\min}, C_{\max}$ are copulas.\\
A particularly interesting subclass of copulas for our problems are so-called shuffles of $M$, see \cite[Section 3.2.3]{Nelsen}.
They represent a construction principle that generates new copulas by means of a
suitable rearrangement of the mass distribution of the upper Fr\'echet bound M.
\begin{defn}[Shuffles of $M$]\label{shuf}
Let $n \geq 1$, $s = (s_0, \ldots, s_n)$ be a partition of the unit interval with $0 = s_0 < s_1 < \ldots < s_n = 1$, $\pi$
be a permutation of $S_n = \{1,\ldots,n\}$ and $\omega \colon S_n \rightarrow \{-1, 1\}$. We define the partition
$t= (t_0, \ldots, t_n), ~0 = t_0 < t_1 < \ldots < t_n = 1$ such that each $[s_{i-1}, s_i[ \times [t_{\pi(i)-1}, t_{\pi(i)}
[$ is a square. A copula $C$ is called shuffle of $M$ with parameters $\{n, s, \pi, \omega\}$ if it is defined in
the following way: for all $i \in \{1, \ldots, n\}$ if $\omega(i) = 1$, then $C$ distributes a mass of $s_i - s_{i-1}$
uniformly spread along the diagonal of $[s_{i-1}, s_i[ \times [t_{\pi(i)-1}, t_{\pi(i)}[$ and if $\omega(i) = -1$
then $C$ distributes a mass of $s_i - s_{i-1}$ uniformly spread along the antidiagonal of $[s_{i-1}, s_i[ \times [t_{\pi(i)-1}, t_{\pi(i)}[$.
\end{defn}
Note that the two Fr\'echet-Hoeffding bounds $W, M$ are trivial shuffles of $M$ with parameters $\{1, (0,1), (1), -1\}$ and $\{1, (0,1), (1), 1\}$,
respectively. Furthermore it is well-known that every copula can be approximated arbitrarily close with respect to the supremum norm by a
shuffle of $M$; see e.g.\ \cite[Theorem 3.2.2]{Nelsen}.\\
In \cite[Theorem 4]{Durante_Sarkoci_Sempi} shuffles of $M$ are characterized in terms of measure preserving transformations $T$ of $[0,1]$
and the push-forward of the doubly stochastic measure induced by $M$. More precisely, the authors proved the following result.
\begin{theorem}
Let $\gamma_C$ denote the doubly stochastic measure induced by the copula $C$ and $\gamma_M$  be the doubly stochastic measure induced by $M$.
The following statements are equivalent:
\begin{itemize}
\item[(a)] a copula $C$ is a shuffle of $M$;
\item[(b)] there exists a piecewise continuous measure-preserving permutation such that $\gamma_C=S_T\ast\gamma_M $, where $S_T:[0,1]^2\rightarrow [0,1]^2$ is defined as $S_T(u,v)=(T(u),v)$ for every $(u,v)\in [0,1]^2$.
\end{itemize}
\end{theorem}
On the other hand, the general problem is that of determining curves in the unit square  which can be considered as the support
of a copula. In \cite{MST} it has been proven that, for every copula obtained as a shuffle of $M$, there is a piecewise linear
function whose graph supports the probability mass. In this context, the following general result holds.
\begin{proposition}
Let $f : [0,1]\rightarrow [0,1]$ be a Borel measurable function. Then, there exists a copula $C$ whose associated
measure $\gamma$ has its mass concentrated on the graph of $f$ (with $\gamma(G (f )) = 1)$ if, and only if, the
function $f$ preserves the Lebesgue measure $\lambda$.
\end{proposition}
These results provide an interesting link between the theory of copulas and the theory of uniform distribution of sequences of points.
In particular, shuffles of $M$ can be considered as special uniform distribution preserving (u.d.p.) mappings.
Recently, different concepts of convergence for copulas have been used;  the question arising in this context
is about the closure of the class of shuffles of $M$ with respect to different notions of convergence and thus topologies.
This problem has been considered for instance in \cite{DuranteSanchez2012}.\\
\\
Now we consider the connection between copulas and uniformly distributed sequences of points in $[0,1[$.
For a detailed account on this concept the interested reader is referred to Drmota \& Tichy \cite{DrmotaTichy1997}.
\begin{defn}
A sequence $(x_n)_{n \in\mathbb{N}}$ of points in $[0,1[$ is called uniformly distributed (u.d.) if and only if
\begin{equation*}
 \lim_{N \rightarrow \infty} \frac{1}{N} \sum_{n = 1}^N \mathbf{1}_{[a, b[} (x_n) = b-a
\end{equation*}
for all intervals $[a, b[ \subseteq [0,1[$, where $\mathbf{1}_E$ denotes as usual the indicator function of the set $E$.
\end{defn}
We call $C$ the asymptotic distribution function (a.d.f.) of a point sequence $(x_n,y_n)_{n \in\mathbb{N}}$ in $[0,1[^2$ if
\begin{equation*}
 C(x,y) = \lim_{N \rightarrow \infty} \frac{1}{N}  \sum_{n = 1}^{N} \mathbf{1}_{[0,x[ \times [0,y[}(x_n,y_n),
\end{equation*}
holds in every point $(x,y)$ of continuity of $C$. In \cite{FialovaStrauch} Fialov\'a and Strauch consider
\begin{equation*}
 \limsup_{N \rightarrow \infty} \frac{1}{N} \sum_{n = 1}^N f(x_n, y_n),
\end{equation*}
where $(x_n)_{n > 1}, (y_n)_{n > 1}$ are u.d.\ sequences in the unit interval and $f$ is a continuous function on
$[0,1]^2$. In this case the a.d.f.\ $g$ of $(x_n,y_n)_{n > 1}$ is always a copula and we can write
\begin{equation}\label{inte}
 \lim_{N \rightarrow \infty} \frac{1}{N} \sum_{n = 1}^N f(x_n, y_n) = \int_0^1 \int_0^1 f(x,y) dg(x,y).
\end{equation}
The following theorem proved in Tichy \& Winkler \cite{TichyWinkler} is of particular relevance for us in order to show the
connection between uniform distribution preserving maps and approximations by means of shuffles of $M$ \cite{MST}.
\begin{theorem}
The set of all continuous piecewise linear u.d.p. mappings are dense in the set of all continuous u.d.p. mappings with respect to uniform convergence. 
\end{theorem}

\section{Theory of optimal transport}\label{transport}
In this section we briefly state fundamental results from the theory of the Monge-Kantorovich optimal transport problem.
The presented results on existence of optimizers and the dual problem formulation are stated in an adequate depth
such that the number theoretic questions under consideration are covered. Comprehensive accounts on the subject are 
Villani \cite{Villani2009} or Rachev \& R\"uschendorf \cite{RachRuesch98} for example, a set of lecture notes on the topic would be
Ambrosio \& Gigli \cite{AmbrosioGigli}.\\
For the basic presentation of the optimal transport problem we are following the above mentioned references.\\
\subsection{Problem formulation}
Let $X$ and $Y$ be Polish spaces and denote $\mathcal{P}(X)$ the set of all Borel probability measures on $X$ ($\mathcal{P}(Y)$ on $Y$).
For a Borel-measurable map $T:X\to Y$ and $\mu\in\mathcal{P}(X)$, the measure $T_{\#}\mu\in\mathcal{P}(Y)$ defined by
\begin{align*}
T_{\#} \mu(E)=\mu(T^{-1}(E))\quad\mbox{for}\;E\in\mathcal{B}(Y)
\end{align*}
is called the \emph{push forward of} $\mu$ \emph{through} $T$. Together with a Borel measurable cost function
$c:X\times Y\to\R\cup\{+\infty\}$ we can give:\\
\\
\underline{Monge formulation}: Let $\mu\in\mathcal{P}(X)$ and  $\nu\in\mathcal{P}(Y)$ and minimize
\begin{align}\label{eq:monge}
\int_X c(x,T(x))\mu(dx)
\end{align}
among all \emph{transport maps} $T$ from $\mu$ to $\nu$ (all maps for which $T_{\#} \mu=\nu$).
The transport map $T$ has the meaning that a unit mass is put from the point $x$ to the point $y=T(x)$ and at the
same time the distribution $\nu$ is achieved.\\
The relaxation of Monge's optimal transport problem is the following:\\
\\ 
\underline{Kantorovich formulation}: Let $\mu\in\mathcal{P}(X)$ and  $\nu\in\mathcal{P}(Y)$ and minimize
\begin{align}\label{eq:OpTrans}
\int_{X\times Y}c(x,y)\gamma(dx,dy)
\end{align}
in the set of $ADM(\mu,\nu)$ of all \emph{transport plans} $\gamma\in\mathcal{P}(X\times Y)$ from $\mu$ to $\nu$, i.e.,
the set of all Borel probability measures on $X\times Y$, such that
\begin{align*}
\gamma(A\times Y)=\mu(A) & \quad \forall A\in\mathcal{B}(X),\\
\gamma(X\times B)=\nu(B) & \quad \forall B\in\mathcal{B}(Y).\\
\end{align*}
Now the meaning of a transport plan is that $\gamma(A\times B)$, for $A\in\mathcal{B}(X)$ and $B\in\mathcal{B}(Y)$, denotes the mass initially placed in $A$
which is moved into $B$, in contrast to transport maps a unit mass can now be split. 
One can immediately answer the question of existence of an optimizer.
\begin{theorem}[Th. 1.5 from \cite{AmbrosioGigli}]
Assume $c$ is lower semicontinuous, then there exists a minimizer for problem (\ref{eq:OpTrans}).
\end{theorem}
The following theorem is an excerpt from Theorem 5.10 of Villani \cite{Villani2009}, it will give us the right tools for the
construction of optimal solutions for some particular examples. 
But at first we need some notions from \cite{RachRuesch98,Villani2009}.
\begin{defn}[$c$-cyclical monotonicity]
A set $\Gamma\subset X\times Y$ is $c$-cyclically monotone if for all $(x_n,y_n)\in\Gamma$ with $1\leq n\leq N$ for $N\in\N$, $x_{N+1}=x_1$, it holds that
\begin{align*}
\sum_{n=1}^N c(x_n,y_n)\leq \sum_{n=1}^N c(x_{n+1},y_n).
\end{align*}
\end{defn}
\begin{defn}[$c$-convexity]
Let $X,\,Y$ be sets and $c:X\times Y\to\R\cup\{+\infty\}$, a function
$f:X\to\R\cup\{+\infty\}$ is $c$-convex if it is not identically $+\infty$ and there exists $a:Y\to\R\cup\{\pm\infty\}$ such that
$$f(x)=\sup_{y\in Y}[a(y)-c(x,y)]\quad\mbox{for all}\;x\in X.$$
Its $c$-transform is defined by
$$f^c(y)=\inf_{x\in X}[f(x)+c(x,y)]\quad\mbox{for all}\;y\in Y,$$
and its $c$-subdifferential is the set
$$\delta_c f=\{(x,y)\in X\times Y\,\vert\,f^c(y)-f(x)=c(x,y)\},$$
or at given $x\in X$
$$\delta_c f(x)=\{y\in Y\,\vert\,f^c(y)-f(x)=c(x,y)\}.$$
\end{defn}
Now everything is clarified such that we can state the theorem.
\begin{theorem}[Th. 5.10 ii) \cite{Villani2009}]\label{th:fundamental}
Let $X$ and $Y$ be Polish spaces, $\mu\in\mathcal{P}(X)$ and $\nu\in\mathcal{P}(Y)$.
Let the cost function $c:X\times Y\to\R\cup\{+\infty\}$
be lower semicontinuous such that $c(x,y)\geq a(x)+b(y)$ for all $(x,y)\in X\times Y$ for some real-valued upper semicontinuous
functions $a\in L^1(\mu)$ and $b\in L^1(\nu)$. Furthermore assume that (\ref{eq:OpTrans}) is finite, then there is
a measurable $c$-cyclically monotone set $\Gamma\subset X\times Y$ such that for any $\gamma\in ADM(\mu,\nu)$
the following statements are equivalent:
\begin{itemize}
\item[(a)] $\gamma$ is optimal,
\item[(b)] $\gamma$ is concentrated on a $c$-cyclically monotone set,
\item[(c)] there is a $c$-convex function $f$ such that, $\gamma$-a.s. $f^c(y)-f(x)=c(x,y)$,
\item[(d)] $\gamma$ is concentrated on $\Gamma$.
\end{itemize}
\end{theorem}
Another variant - or better formulation - of problem (\ref{eq:OpTrans}) is through a coupling of random variables.
This formulation is of more probabilistic nature and perfectly suits applications from mathematical finance.\\
\\
\underline{Coupling formulation}: At first the \textquoteleft$\inf$\textquoteright{} from (\ref{eq:OpTrans}) is turned
into a \textquoteleft$\sup$\textquoteright{}, according to a standard presentation from R\"uschendorf \cite{Rueschen}.
Secondly instead of \emph{transport plan} the notion \emph{coupling} is used, i.e. coupling of two random variables with fixed marginal distributions.
Then one may write problem (\ref{eq:OpTrans}) as, cf.\cite[Sec. 4.2]{Rueschen},
\begin{align}\label{eq:coupling}
\sup\{\erw(c(X_1,X_2))\,\vert\,X_1,\,X_2\,\mbox{couplings of}\,\mu,\nu\,\mbox{with}\,P_{X_1}=\mu\,\mbox{and}\,P_{X_2}=\nu\}.
\end{align}
The supremum is taken over the set of all bivariate $X\times Y$-valued random variables $(X_1,X_2)$
with given marginal distributions $\mu$ and $\nu$. Consequently, $\gamma\in ADM(\mu,\nu)$ corresponds to the bivariate distribution of $(X_1,X_2)$.\\
\\
For translating Theorem \ref{th:fundamental} to this maximization situation one needs to adapt the $c$-convexity notion.
Now a function $f:X\to\R$ is called $c$-convex if it has a representation $f(x)=\sup_y\{c(x,y)+a(y)\},$ for some function $a$.
The $c$-subdifferential of $f$ at $x$ is directly (hiding the $c$-transform) given by
$$\partial_c f(x)=\{y\,\vert\,f(z)-f(x)\geq c(z,y)-c(x,y)\;\forall\,z\in X\}$$
and $\partial_c f=\{(x,y)\in X\times Y\,\vert\,y\in\partial_c f(x)\}$.\\
Under the assumptions of Theorem \ref{th:fundamental} for $X,\,Y$ and lower semicontinuity of $c$ one gets an analogous statement.
\begin{theorem}[Th. 4.7 from \cite{Rueschen}]\label{Th:optimality}
Let $c$ be such that $c(x,y)\geq a(x)+b(y)$ for some $a\in L^1(\mu)$, $b\in L^1(\nu)$) and assume finiteness of (\ref{eq:coupling}).
Then a pair $(X_1,X_2)$ with $X_1\stackrel{d}{\sim}\mu$, $X_2\stackrel{d}{\sim}\nu$ is an optimal $c-$coupling between $\mu$ and $\nu$
if and only if $$(X_1,X_2)\in\partial_c f\;\mbox{a.s.}$$ for some $c$-convex function $f$, equivalently, $X_2\in\delta_c f(X_1)$ a.s.
\end{theorem}
\underline{Uniform Marginals}\\
In the following sections we will apply Theorem \ref{th:fundamental} to problems in the spirit of \cite{FialovaStrauch}
and in particular to problem (\ref{problem}). Therefore in these applications we choose $X=Y=[0,1]$,
fix the marginal distributions to be uniform $\mu=\nu=\mathcal{U}[0,1]$ and choose the cost function $c:[0,1]^2\to\R$ to be continuous.
For this setting Theorem \ref{th:fundamental} applies and (\ref{eq:OpTrans}) is clearly finite. In this setting
the Monge problem is an optimization problem with respect to uniform distribution preserving maps.
\begin{remark}
In general the question if an optimal transport plan in the Kantorovich formulation (\ref{eq:OpTrans}) is induced by an optimal transport
map from the Monge formulation \eqref{eq:monge} is not answered in the corresponding literature.
But there are some particular situations which allow for affirmative results. For instance for Borel measures $\mu,\,\nu$ on $\R^d$, $c(x,y)=h(x-y)$
for a strictly convex function $h$ and $\mu$ absolutely continuous with respect to Lebesgue measure there exists an unique map $s$ such
that the measure $\gamma=(\mathbf{id}\times s)_{\#}\mu$ is optimal. For the situation of one dimensional marginals several
examples of a similar structure are given in Uckelmann \cite{Uckelmann} and R{\"u}schendorf \& Uckelmann \cite{RueschUckel}.\\
Theorem \ref{th:fundamental} or its complete formulation which is Theorem 5.10 from Villani \cite{Villani2009} heavily rely on a dual problem formulation.
How far the duality relation can be exploited, i.e., up to which parameter configuration the value of the primal solution equals the value of dual solution, is 
recently discussed in Beiglb\"ock \& Schachermayer \cite{BeiglSchach2011} and Beiglb{\"o}ck et al. \cite{Beigletal2012}.
\end{remark}

\subsection{Application of  Theorem \ref{th:fundamental}}
Now we are ready to give two applications of Theorem \ref{th:fundamental}.
The first one, dealing with a result from \cite{FialovaStrauch}, concerns a direct verification of optimality via
the $c$-cyclical monotonicity property. The second one treats an open problem from \cite{Prob129} by constructing
explicitly a $c$-convex function.\\
\\
We start by studying Theorem 5. from Fialov{\'a} \& Strauch \cite{FialovaStrauch} from the transport point of view.

\begin{theorem}[\cite{FialovaStrauch}]
Let $c:[0,1]^2\to\R$ be a Riemann integrable function with $\frac{\partial^2 c}{\partial x\partial y}>0$ for all $(x,y)\in(0,1)^2$.
Then
\begin{align*}
\max_{\gamma\, -\,\mbox{copula}} \int_0^1\int_0^1 c(x,y)\gamma(dx,dy)&=\int_0^1 c(x,x)dx,\\
\min_{\gamma\, -\,\mbox {copula}} \int_0^1\int_0^1 c(x,y)\gamma(dx,dy)&=\int_0^1 c(x,1-x)dx,
\end{align*}
where the maximum is attained in $\gamma^u=\min\{x,y\}$ and the minimum in $\gamma^l(x,y)=\max\{x+y-1,0\}$, uniquely.
\end{theorem}
\begin{remark}
The maximum and the minimum are attained at the co-called upper and lower Fr\'{e}chet bounds.\\
Now by means of the $c$-cyclical monotonicity criterion for optimality we can give an alternative proof of the Theorem above.
The procedure is similar to the one used in the proof of Proposition 1. from Rochet \cite{Rochet} in a slightly different context.
Rochet proves that if the derivative condition is fulfilled the support of a transport plan $\Gamma(x)$ is $c$-cyclically monotone if and only if
$\Gamma(\cdot)$ is non-decreasing. Combining this with a result on uniform distribution preserving maps from
Tichy \& Winkler \cite{TichyWinkler} one arrives at Fialov{\'a} \& Strauch's result that $\Gamma(x)=x$
is the maximizing transport plan and $\Gamma(x)=1-x$ is the minimizing one.\\
Notice Theorem 3.1.2.\ from \cite{RachRuesch98} includes the situations of Theorem 5. and 6. from Fialov{\'a} \& Strauch \cite{FialovaStrauch}.
\end{remark}
\begin{proof}
We need to show that for $N\in\N$ and $x_0,\,x_1,\ldots,x_N$
\begin{align*}
 \sum_{n=0}^N c(x_{n+1},y_n)-c(x_n,y_n)\leq 0,
\end{align*}
where $x_{N+1}=x_0$ and $(x_n,y_n)$ are in the support of the to be shown optimal measure. In our particular situation for
the upper bound $y_n=\Gamma(x_n)=x_n$.
We start with a cycle of length 2 ($N=1$) and points $x_0,\,x_1$ and assume w.l.o.g $x_1>x_0$. From
\begin{align*}
c(x_{0},x_{1})-c(x_1,x_{1})+c(x_{1},x_0)-c(x_0,x_0)=-\int_{x_0}^{x_{1}}\int_{x_0}^{x_{1}}\frac{\partial^2 c}{\partial
x\partial y}(x,y)dydx\leq 0
\end{align*}
the statement follows. Now assume that for all cycles of length $N$ the statement holds true,
\begin{align*}
\sum_{n=0}^{N-1} c(x_{n+1},x_n)-c(x_n,x_n)\leq 0,
\end{align*}
and choose arbitrary $x_0,\,x_1,\ldots,x_N$ (again w.l.o.g. $x_N=\max_{n\in\{0,1,\ldots,N\}} x_n$).
Fix a sub-cycle of length $N-1$ with $x'_n=x_n$ for $0\leq n\leq N-1$ and $x'_N=x_0$. Then using 
\begin{align*}
A:=\sum_{n=0}^{N-1} c(x_{n+1}',x_n')-c(x_n',x_n')\leq 0,
\end{align*}
we get ($x_{N+1}=x'_{N}=x_0$),
\begin{align*}
\sum_{n=0}^{N} c(x_{n+1},x_n)-c(x_n,x_n)=&A+[c(x_N,x_{N-1})-c(x_{N-1},x_{N-1})]-[c(x_0,x_{N-1})-c(x_{N-1},x_{N-1})]\\
&+[c(x_{N+1},x_N)-c(x_N,x_N)]\\
\leq& c(x_N,x_{N-1})-c(x_0,x_{N-1})+c(x_{0},x_N)-c(x_N,x_N)\\
=&-\int_{x_0}^{x_N}\int_{x_{N-1}}^{x_N}\frac{\partial^2 c}{\partial x\partial y}(x,y) dydx\leq 0,
\end{align*}
which completes the proof.
\end{proof}
{\bf{The sine example}}\\
Now we can turn our view on a particular result from Uckelmann \cite{Uckelmann}, see R{\"u}schendorf \& Uckelmann \cite{RueschUckel}
as well, which somehow perfectly matches the sine question. It is an open problem from \cite{Prob129}) which could not be treated by
direct analytical methods.\\
For dealing with this example one may utilize a modification of Theorem 1. from \cite{Uckelmann}.
We re-state the following version of this result and sketch its proof.
\begin{remark}
In this particular setting it is possible to construct explicitly an $c$-convex function $f$ such that Theorem \ref{Th:optimality} can be used
to deduce the optimal transport plan, i.e. optimal coupling. Notice, see \cite{Rueschen}, that $y\in\partial_c f(x)$
if and only if $\exists\,a(=a(y))\in\R$ such that
\begin{align}\label{eq:ubdiffeasy}
\psi_{y,a}(x)=c(x,y)+a(y)=f(x)\quad\mbox{and}\quad \psi_{y,a}(\xi)=c(\xi,y)+a(y)\leq f(\xi)\quad\forall\,\xi\in X.
\end{align}
\end{remark}

\begin{theorem}
Let $\mu,\,\nu$ be the uniform distribution on $[0,1]$ and the cost function $c(x,y)=\phi(x+y)$ with 
$\phi:[0,2]\to\R$. In particular we assume that $\phi\in\mathcal{C}^2[0,1]$ and that there is $k\in(0,2)$
such that $\phi''(x)<0$ for $x\in[0,k)$ and $\phi''(x)>0$ for $x\in(k,2]$. If $\beta\in(0,1)$ denotes the solution to
$$\phi(2\beta)-\phi(\beta)=\beta\phi'(\beta),$$
then
\begin{align*}
\Gamma(x)=
\left\{\begin{array}{cc}
\beta-x,& x\in[0,\beta),\\
x,&x\in[\beta,1],
\end{array}\right.
\end{align*}
induces by $(U,\Gamma(U))$ for some standard uniformly distributed $U$ an optimal $c$-coupling between $P$ and $Q$.
\end{theorem}
\begin{proof}
For the proof we proceed as proposed in \cite{Uckelmann} and \cite{RueschUckel}.
Define the following functions:
\begin{align*}
f_1(x)&=x\phi'(\beta),\\
f_2(x)&=\frac{1}{2}(\phi(2 x)-\phi(2\beta))+\beta \phi'(\beta),\\
\psi^1(\xi)&=\phi(\beta-x+\xi)+x\phi'(\beta)-\phi(\beta),\\
\psi^2(\xi)&=\phi(x+\xi)-\frac{1}{2}\phi(2x)-\frac{1}{2}\phi(2\beta)+\beta\phi'(\beta).
\end{align*}
Furthermore set
$$f(x)=f_1(x)I_{[0,\beta)}(x)+f_2(x) I_{[\beta,1]}(x),$$
and put for $\xi\in[0,1]$:
\begin{align*}
\psi_{\Gamma(x)}(\xi)=
\left\{\begin{array}{cc}
\psi^1(\xi),& x\in[0,\beta),\\
\psi^2(\xi),& x\in[\beta,1].
\end{array}\right.
\end{align*}
Here $\psi_{\Gamma(x)}(\xi)$ plays the role of $\psi_{y,a}(\xi)=c(\xi,y)+a(y)$ with $y=\Gamma(x)$ in (\ref{eq:ubdiffeasy}).
Now the idea, following Theorem \ref{Th:optimality} and (\ref{eq:ubdiffeasy}),
is to show that $y=\Gamma(x)$ is in the $c$-subdifferential of $f(x)$ for all $x\in[0,1]$ which implies optimality
of this particular coupling and optimality of the distribution induced by $(U,\Gamma(U))$ for the transport problem.
For the $c$-convexity of $f$ and the subdifferential property we need to show:
\begin{align*}
\psi_{\Gamma(x)}(x)&=f(x)\quad\forall\,x\in[0,1],\\
\psi_{\Gamma(x)}(\xi)&\leq f(\xi)\quad\forall\,\xi\in[0,1].
\end{align*}
We start with showing that $\psi_{\Gamma(x)}(x)=f(x)$. For $x\in[0,\beta)$ we have that $\Gamma(x)=\beta-x$ and
$$\psi_{\Gamma(x)}(x)=\psi^1(x)=x\phi'(\beta)=f_1(x)=f(x).$$
For $x\in[\beta,1]$ we have $\Gamma(x)=x$ and
$$\psi_{\Gamma(x)}(x)=\psi^2(x)=\frac{1}{2}(\phi(2 x)-\phi(2\beta))+\beta\phi'(\beta)=f_2(x)=f(x).$$
It remains to show $\psi_{\Gamma(x)}(\xi)\leq f(\xi)$ for all $(x,\xi)\in[0,1]\times [0,1]$, which can be achieved by a rather lengthy and carefull analysis,
for the details see \cite{Strauch2015}.
\end{proof}
\begin{remark}
If $\beta>1$ then it can be shown as in the first step of the above proof that $(U,1-U)$ yields the optimal coupling.
Loosely speaking one could say that the concave behaviour dominates the convex one.
\end{remark}
Now we are prepared to answer the sine question. Setting $\phi(z)=\sin(\pi z)$ and $k=1$ we immediately get:
\begin{kor}
For $c(x,y)=\sin(\pi(x+y))$ we have that the distribution of the vector $(U,\Gamma(U))$ for $U\sim\mathcal{U}([0,1])$ with
\begin{align*}
\Gamma(x)=\left\{\begin{array}{cc}
\beta-x,& x\in[0,\beta),\\
x,&x\in[\beta,1],
\end{array}\right.
\end{align*}
and $\beta=0.7541996008265638\approx0.7542$ which solves
\begin{align}\label{eq:1storder}
\sin(2\pi\beta)-\sin(\pi\beta)=\beta\pi\cos(\pi\beta),
\end{align}
is maximizing
$$\int_{[0,1]^2} \sin(\pi(x+y)) d\gamma(x,y)$$
in the set of all bivariate distributions $\gamma$ with uniform marginals, i.e., in the set of all copulas.
\end{kor}
\begin{remark}
In this situation equation (\ref{eq:1storder}) meets the first order condition when looking at couplings of the form
$(U,\Gamma^{\alpha}(U))$ with 
\begin{align*}
\Gamma^{\alpha}(x)=\left\{
\begin{array}{cc}
\alpha-x,& x\in[0,\alpha),\\
x,& x\in[\alpha,1],
\end{array}\right.
\end{align*}
or explicitly maximizing ($c(x,y)=\sin(\pi(x+y)$)
\begin{align*}
H(\alpha):=\int_0^{\alpha }c(x,\alpha -x) \, dx+\int_{\alpha }^1 c(x,x) \, dx.
\end{align*}
\end{remark}
\section{Approximations}
In this section we are going to introduce some implementable approximation methods for the optimal transport problem.
Since the computational methods are based on the \emph{assignment problem}, see Burkard et al. \cite{BDM}, we
recapitulate it and mention (some) one of the fundamental numerical solution algorithm(s). Some connections between
the particular \emph{copula maximization} problem and the assignment problem are already given in \cite{HoferIaco}.
There the authors showed that for a piecewise constant cost function the maximizing measure is induced by a
shuffle of $M$ whose parameters are linked to the permutation which solves the corresponding assignment problem.\\
\\
The (linear sum) assignment problem from combinatorial optimization is given by a matrix $(c_{ij})_{1\leq i,j\leq n}$
with entries $c_{ij}\in\R$ which represent costs when assigning $j$ to $i$, or when transporting 1 unit of mass from
$i$ to $j$. The goal is to match each row to a different column at minimal cost,
\begin{align*}
\sum_{i=1}^n\sum_{j=1}^n c_{ij}x_{ij}\;\to\;\mbox{minimum}
\end{align*}
under the constraints $\sum_{j=1}^n x_{ij}=1$ for $i\in\{1,\ldots,n\}$, $\sum_{i=1}^n x_{ij}=1$ for $j\in\{1,\ldots,n\}$
and $x_{ij}\in\{0,1\}$ for all $i,\,j\in\{1,\ldots,n\}$. An interesting remark is given in \cite[p. 75]{BDM}, it states
that the problem above is equivalent to its continuous relaxation with $x_{ij}\geq 0$ for all $i,\,j\in\{1,\ldots,n\}$.\\
\\
Notice that if specifying $\mu=\frac{1}{n}\sum_{i=1}^n\delta_{x_i}$, $\nu=\frac{1}{n}\sum_{j=1}^n\delta_{y_j}$ for points $\{x_1,\ldots,x_n\}$,
$\{y_1,\ldots,y_n\}$ in the unit interval and identifying $c(x_i,y_j)=c_{ij}$ one recovers exactly the assignment problem from
the original transport problem. This connection also represents the first step in the proof of the fundamental Theorem 5.20 from \cite{Villani2009}.
\begin{remark}
As mentioned above a standard reference, from the theoretical as well as from the algorithmic point of view, is Burkard et al. \cite{BDM}.
Another reference for so-called quadratic assignment problems, incorporating a different cost structure, is {\c{C}}ela \cite{Cela98}.
In the paper by Beiglb\"ock et al. \cite{BeiglLeoSchach2014} the general duality theory of the transport problem
is motivated by an explicit study of the assignment problem.
\end{remark}

The assignment problem was among the first linear programming problems to
be studied extensively. Given $n$ workers and $n$ jobs, we know, for every job, the salary that should be paid to each worker for him to perform the job. The goal is to find the the best assignment, i.e. each worker is assigned to exactly one job and vice versa in order to minimize the total cost (the sum of the salaries).\\
One of the many algorithms which solves the linear assignment problem is due to Kuhn \cite{Kuhn} and Munkres \cite{Munkres} (also known as the Hungarian method).\\
In fact the algorithm was developed and published by Kuhn \cite{Kuhn}, who gave the name \lq\lq Hungarian method\rq\rq.
Munkres \cite{Munkres} reviewed the algorithm and observed that it is strongly polynomial.\\
The problem is formulated as follows: given $n$ workers and tasks, and an $n\times n$ matrix containing the cost of assigning each worker to a task, find the cost minimizing assignment.

First the problem is written in the following matrix form
\begin{equation*}
\begin{pmatrix}
a_{11} & a_{12} & \dots & a_{1n} \\
a_{21} & a_{22} & \dots & a_{2n}\\
\vdots & \ddots & & \\
a_{n1} & a_{n2} & \dots & a_{nn}
 \end{pmatrix}\ ,
\end{equation*}
where the $a_{i,j}$'s denote the penalties incurred when worker $i$ performs task $j$.\\
The first step of the algorithm consists in subtracting the lowest $a_{i,j}$'s of the $i$th row from each element in that row. This will lead to at least one zero in that row. This procedure is repeated for all rows. We now have a matrix with at least one zero per row. We repeat this procedure for all columns (i.e. we subtract the minimum element in each column from all the elements in that column).\\
Then, we draw lines through the rows and columns so that all the
zero entries of the matrix are covered and the minimum number
of such lines is used. Finally, we check if an optimal assignment is possible. If the minimum number of covering lines is exactly $n$, an optimal assignment of zeros is possible and we are finished. Otherwise, 
if the minimum number of covering lines is less than $n$, an
optimal assignment of zeros is not yet possible and we determine the smallest entry not covered by any line. We again subtract this
entry from each uncovered row, and then add it to each covered
column and we perform again the optimality test.\\

Due to its several applications and to the possible connections with related problems, many researchers got interested in the higher dimensional version of the linear assignment problem, the MAP, where one aims to find tuples of elements from given sets, such that the total cost of the tuples is minimal. While the linear assignment problem is solvable in polynomial time, the MAP is NP-hard (see e.g. \cite{BDM}). Recently, a new approach based on the Cross-Entropy (CE) methods has been developed in \cite{Nguyen_et_Al2014} for solving the MAP. The efficiency of this method is corroborated by several teCsts on large-scale problems.
\subsection{Theoretical basis}
For computational purposes the following fairly general result due to Schachermayer \& Teichmann \cite{SchachTeich} is valuable.
\begin{theorem}[Th. 3 from \cite{SchachTeich}]\label{th:approxST}
Let $c:X\times Y\to\R_{\geq 0}$ be a finitely valued, continuous cost function on Polish spaces $X,\,Y$. Let $\{\pi_n\}_{n\geq 0}$
be an approximating sequence of optimizers associated to weakly converging sequences $\mu_n\to\mu$ and $\nu_n\to\nu$ as $n\to\infty$,
i.e., $\pi_n$ being an optimizer for the transport problem $(c,\mu_n,\nu_n)$. Then there is
a subsequence $\{\pi_{n_k}\}_{k\geq 0}$ converging weakly to a transport plan $\pi$ on $X\times Y$, which optimizes
the Monge-Kantorovich problem for $(\mu,\nu,c)$. Any other converging subsequence of $\{\pi_n\}_{n\geq 0}$ also
converges to an optimizer of the Monge-Kantorovich problem, i.e., the non-empty set of adherence points of $\{\pi_n\}_{n\geq 0}$
is a set of optimizers.
\end{theorem}
An approximation result suited for the uniform marginals situation is Theorem 2.2 from \cite{HoferIaco}, here we can complement it
by proving that a (sub-)sequence of the discrete optimizers converges to an optimizer of the limiting continuous problem.
\begin{theorem}\label{th:approxHI}
Let $c$ be a continuous function on $[0,1]^2$, let the sets $I^n_{i,j}$ be given as
\begin{equation*}
I^n_{i,j}=\left[\frac{i-1}{2^n},\frac{i}{2^n}\right[\times\left[\frac{j-1}{2^n},\frac{j}{2^n}\right[\;\text{for}\;i,j=1,\ldots,2^n,
\end{equation*}
for every $n>1$ and define the functions $\underline{c}_n,\,\overline{c}_n$ as
\begin{align}
\underline{c}_n(x,y) &= \min_{(x,y) \in I^n_{i,j}} c\left( x, y \right), \quad \text{ for all } (x,y) \in I^n_{i,j},\notag\\
\overline{c}_n(x,y) &= \max_{(x,y) \in I^n_{i,j}} c\left( x, y \right), \quad \text{ for all } (x,y) \in I^n_{i,j}.\label{mini}
\end{align}
Furthermore, let $\underline{\gamma}^n_{\max},\,\overline{\gamma}^n_{\max}$ be maximizing measures for cost functions $\underline{c}_n$
and $\overline{c}_n$ respectively. Then
\begin{align}
\lim_{n \rightarrow \infty} \int_{[0,1[^2} \underline{c}_n(x,y) \underline{\gamma}^n_{\max}(dx,dy)
&= \lim_{n \rightarrow \infty} \int_{[0,1[^2} \overline{c}_n(x,y) \overline{\gamma}^n_{\max}(dx,dy) \notag\\
&= \sup_{\gamma \in \mathcal{C}} \int_{[0,1[^2} c(x,y) \gamma(dx,dy).\label{conv}
\end{align}
Furthermore the sequence of maximizers converges, at least along some subsequence, to a maximizer of the original problem
$\int_{[0,1]^2}c\gamma(dx,dy)$.
\end{theorem}
\begin{proof}
The first statement (\ref{conv}) is already shown in \cite{HoferIaco}. For the remaining part we can proceed
in the spirit of Villani's proof of Theorem 5.20 from \cite{Villani2009}, notice there a sequence of continuous
functions $c_n$ is considered. We will show the proof for the lower approximation via $\underline{c}_n$.\\
At first observe that $\underline{\gamma}^n$ converges (at least along some subsequence) to some measure $\gamma^*$ with uniform marginals, cf.
\cite[Thm. 5.21]{Kallenberg}.\\
From above we know that $\underline{\gamma}^n$ is concentrated on a $\underline{c}_n$-cyclically monotone set with the consequence that
for some $N\in\N$ the $N$-fold product measure $\underline{\gamma}^{n,\oplus N}$ is concentrated on the set $\mathcal{S}_n(N)$ of points
$(x_1,y_1),\ldots,(x_N,y_N)$ for which
\begin{align*}
\sum_{j=1}^N \underline{c}_n(x_j,y_j)\geq \sum_{j=1}^N\underline{c}_n(x_{j+1},y_j),
\end{align*}
with $x_{N+1}=x_1$. Now fix some $\varepsilon>0$ and choose $n$ large enough, such that $\underline{\gamma}^{n,\oplus N}$ is concentrated on the set
$\mathcal{S}_{\varepsilon}(N)$ of points with
\begin{align*}
\sum_{j=1}^N c(x_j,y_j)\geq \sum_{j=1}^N c(x_{j+1},y_j)-\varepsilon.
\end{align*}
Since $c$ is continuous we have that $\mathcal{S}_{\varepsilon}(N)$ is a closed set. This (using indicator functions in the weak
convergence characterization)
implies that also the limiting measure $\gamma^{*,\oplus N}$ is concentrated on $\mathcal{S}_{\varepsilon}(N)$ for all $\varepsilon>0$.
We can let $\varepsilon\to 0$ and derive that $\gamma^{*,\oplus N}$ is concentrated on a set of points with
\begin{align*}
\sum_{j=1}^N c(x_j,y_j)\geq \sum_{j=1}^N c(x_{j+1},y_j)
\end{align*}
and therefore is concentrated on a $c$-cyclically monotone set. Since the costs are bounded
we deduce from Theorem \ref{th:fundamental}(b) that $\gamma^*$ is optimal.
\end{proof}
\begin{remark}
Theorem \ref{th:approxST} can be used when approximating $\mu$ and $\nu$ by empirical distributions.
Let $X_1,\,X_2,\ldots$ be i.i.d. random variables with distribution $\mu$ and
$Y_1,\,Y_2,\ldots$ be i.i.d. random variables with distribution $\nu$. Then the empirical distributions defined by
$\hat{\mu}_n(-\infty,x]=\frac{1}{n}\sum_{k=1}^n I_{\{X_k\leq x\}}$ and $\hat{\nu}_n(-\infty,x]=\frac{1}{n}\sum_{k=1}^n I_{\{Y_k\leq x\}}$
converge weakly to $\mu$ and $\nu$, see \cite[Prop. 4.24]{Kallenberg}. In this situation one can solve assignment
problems along realizations of the random sequences $\{X_k\}$ and $\{Y_k\}$.\\
The spirit of Theorem \ref{th:approxHI} is a little different. There the marginal distributions are fixed to be uniform,
whereas the cost function is approximated by piecewise constant functions on a deterministically chosen grid. This particular situation
is linked to a solution of the assignment problem via Theorem 2.1 of \cite{HoferIaco}.
\end{remark}

\subsection{Numerical examples}

An explicit implementation of the Hungarian algorithm applied to the cost function $c(x,y)=\sin(\pi(x+y))$ can be found in \cite{HoferIaco}. The authors provide also a numerical solution to a problem in financial mathematics, namely the First-to-Default (FTD) swap. This is a  contract between an insurance buyer and an insurance seller. The first one makes periodic premium payments, called spreads, until the maturity of the contract or the default, whichever occurs first. In exchange the second one compensates the loss caused by the default at the time of default. In \cite{HoferIaco} an approximation of the value of the maximal spread is provided.\\ Of course other applications are possible and interesting, but this goes beyond the purposes of the present paper.\\
Our aim is to consider some cost functions $c$ involving the sine function as in \cite{HoferIaco} and to show how the support of the copula where the maximum is attained can vary considerably. More precisely, we consider
\begin{equation}\label{num}
\limsup_{N \rightarrow \infty} \frac{1}{N} \sum_{n = 1}^N c (x_n + y_n)\ ,
\end{equation}
with $c(x,y)=\sin(2\pi x)\sin(2\pi y)$, $c(x,y)=\sin(2\pi x)\cos(2\pi y)$ and $c(x,y)=\sin(2\pi/x)\cos(2\pi y)$. In particular, the numerical costs for the transport map in the first case is 0.5, which is the same as when using the identity map (explicitly computed). Let us remark that obviously there is no unique solution.\\
The numerical results are illustrated in Table \ref{tab:numerics}.

\begin{table}[!ht]
\centering
\begin{tabular}{c|c|c|c}
\backslashbox{$n$}{$c$}& $\sin(\pi x)\sin(\pi y)$ & $\sin(\pi x)\cos(\pi y)$ & $\sin(\pi/x)\cos(\pi y)$\\
\hline
2 & 0.5 & 0.1768 & 0.4612\\
3 & 0.5 & 0.2039 & 0.3402\\
4 & 0.5 & 0.2102 & 0.5067\\
5 & 0.5 & 0.2117 & 0.4012\\
6 & 0.5 & 0.2121 & 0.4580\\
7 & 0.5 & 0.2122 & 0.4400
\end{tabular}
\caption{Upper bounds for the $\limsup$ in (\ref{num}).}\label{tab:numerics}
\end{table}

\begin{figure}[h!]
 \centering
 \includegraphics[scale=0.4]{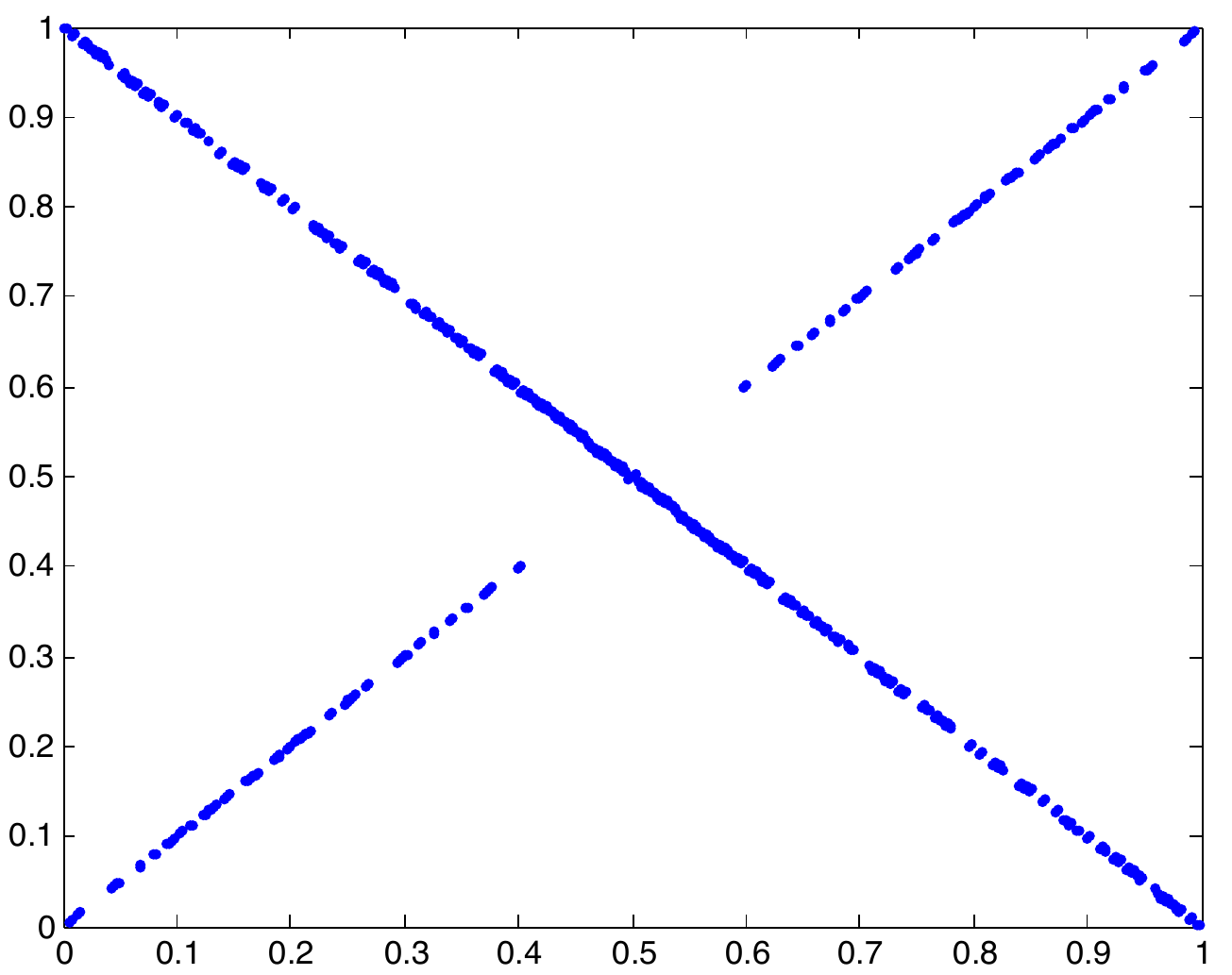}
\caption{Support of copula which attains upper bound for $c(x,y)=\sin(\pi x)\sin(\pi y)$ and $n=10$}
\label{fig:sinsin}
\end{figure}

\begin{figure}[h!]
\begin{center}
\centering\includegraphics[scale=0.25]{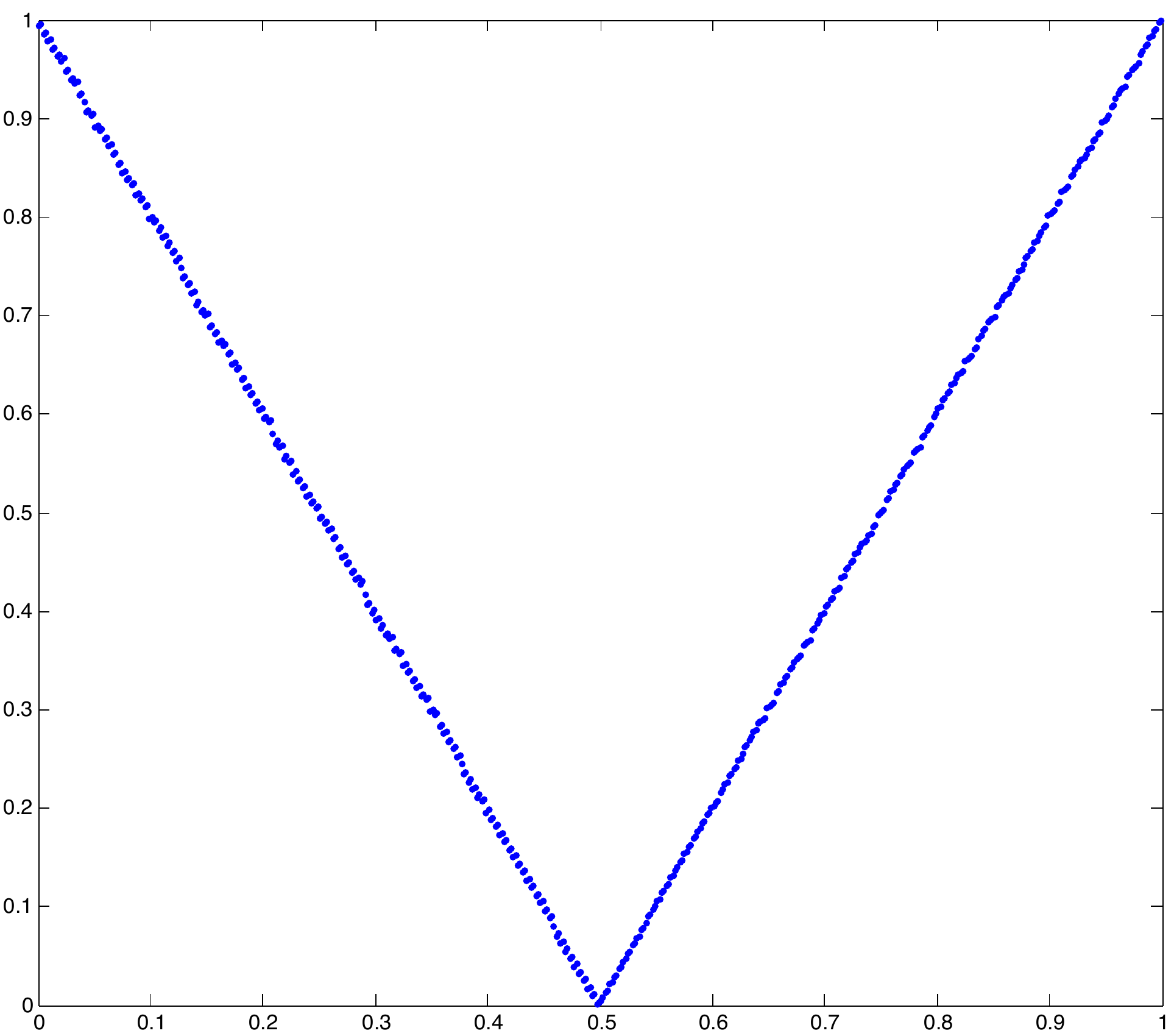}
\caption{Support of copula which attains upper bound for $c(x,y)=\sin(\pi x)\cos(\pi y)$ and $n=10$}\label{fig:sincos}
\end{center}
\end{figure}

\begin{figure}[h!]
\begin{center}
\includegraphics[scale=0.42]{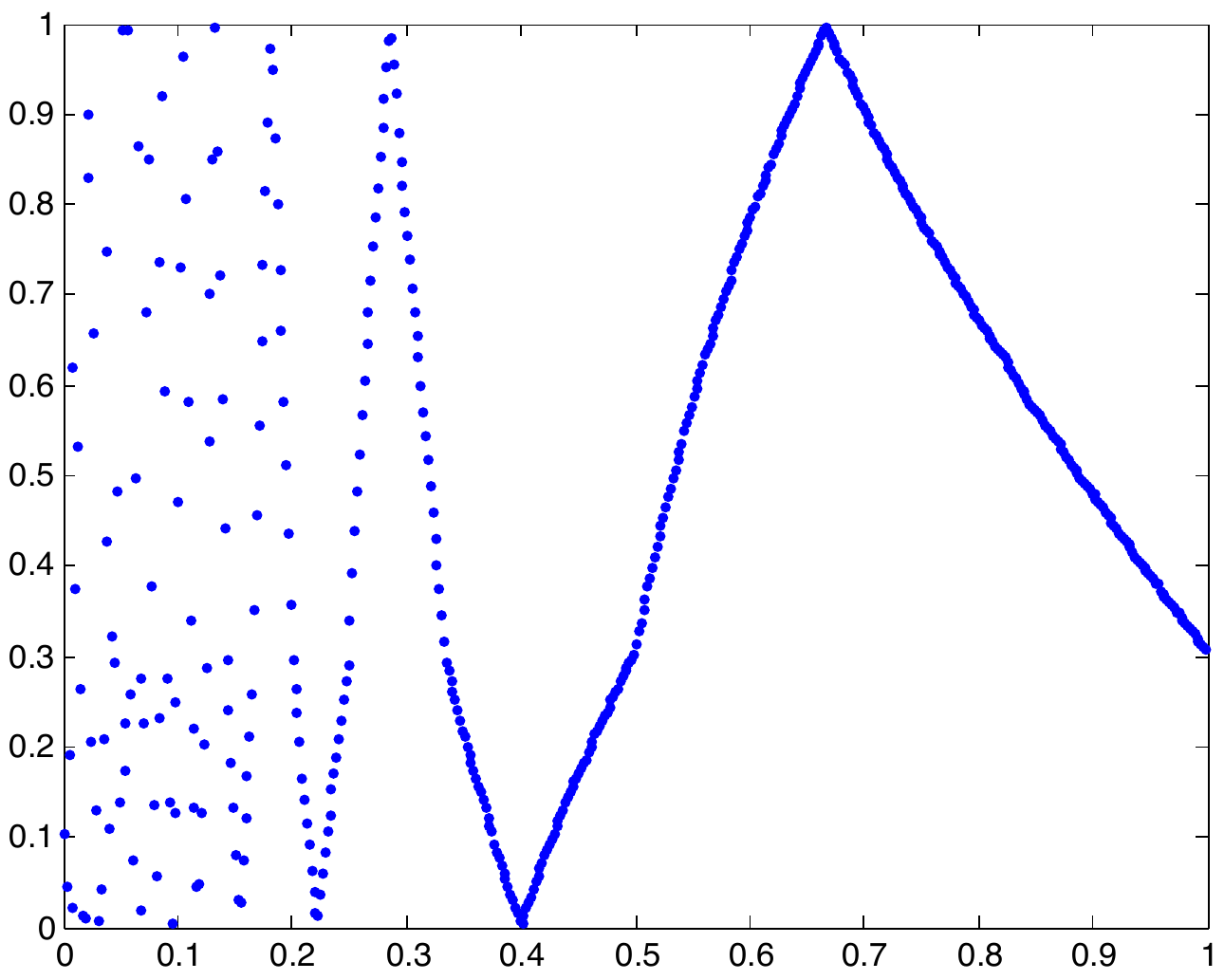}
\caption{Support of copula which attains upper bound for $c(x,y)=\sin(\pi/x)\cos(\pi y)$ and $n=10$}\label{fig:sinchaotic}
\end{center}
\end{figure}

{\footnotesize
\bibliographystyle{abbrv}
\bibliography{itt_bib}}
\Addresses
\end{document}